\documentclass[11pt]{amsproc}
\usepackage{amsfonts}
\usepackage{amsmath,amstext,amsbsy,amssymb}

\newtheorem{theorem}{Theorem}[section]
\newtheorem{lemma}[theorem]{Lemma}

\newtheorem{proposition}[theorem]{Proposition}

\theoremstyle{definition}
\newtheorem{definition}[theorem]{Definition}
\newtheorem{example}[theorem]{Example}

\theoremstyle{remark}
\newtheorem{remark}[theorem]{Remark}

\numberwithin{equation}{section}
\usepackage{color}

\begin{document}
\title{Average length of the longest $k$-alternating subsequence}
\author{Tommy Wuxing Cai}
\address{School of Sciences,
South China University of Technology, Guangzhou 510640, China}
\email{caiwx@scut.edu.cn}

\begin{abstract}
We prove a conjecture of Drew Armstrong on the average maximal length of $k$-alternating subsequence of permutations. The $k=1$ case is a well-known result of Richard Stanley.
\end{abstract}
\maketitle
\section{Introduction}
We fix positive integers $n,k$ with $n\geq2$ and $1\leq k\leq n-1$.

Let $w=w_1w_2\dotsm w_n$ in $\mathfrak{S}_n$, the permutation group of $[1,n]$. A subsequence $w_{i_1}\dotsm w_{i_s}$ of $w$ is \emph{alternating} if $w_{i_1}>w_{i_2}<w_{i_3}\dotsm$. We call it \emph{$k$-alternating} if moreover each neighboring pair satisfies $|w_{i_j}-w_{i_{j+1}}|\geq k$. We call the maximal length (which is the number of elements) of the $k$-alternating subsequences of $w$ the \emph{$k$-alternating length} of $w$ and denote it as $as_k(w)$ \cite{Arm}. We denote the average of the $k$-alternating length of permutations in $\mathfrak{S}_n$ by $E_n(as_k)$; i.e., $E_{n}(as_k)=\frac{1}{n!}\sum_{w\in \mathfrak{S}_n}as_k(w)$. We prove the following result which was conjectured by Drew Armstrong \cite{Arm}:

\begin{theorem}\label{T:main}
For integers $n,k$ with $n\geq k+1\geq2$, the average $k$-alternating length of permutations in $\mathfrak{S}_n$ is
\begin{equation}\label{F:averagekalternating}
E_{n}(as_k)=\frac{4(n-k)+5}{6}.
\end{equation}
\end{theorem}
The special case when $k=1$ is a result of Stanley \cite{S1,S2}. Igor Pak and Robin Pemantle proved that $E_n(as_k)$ is asymptotically $2(n-k)/3$ using a probabilistic method \cite{PP}.

We call a subsequence satisfying $w_{i_1}<w_{i_2}>w_{i_3}\dotsm$ \emph{reverse alternating}.  We say a subsequence is \emph{zigzagging} if it is either alternating or reverse alternating. Then we similarly define a \emph{$k$-zigzagging subsequence} and the \emph{$k$-zigzagging length} $zs_k(w)$. We denote the average $k$-zigzagging length of permutations in $\mathfrak{S}_n$ by $E_n(zs_k)$.

Note that the \emph{swapping map} $I:w_1w_2\dotsm w_n\rightarrow (n+1-w_1)(n+1-w_2)\dotsm (n+1-w_n)$ is an involution interchanging alternating subsequences and reverse alternating subsequences. Thus exactly half of the permutations $w\in\mathfrak{S}_n$ have $k$-zigzagging length that is one more than their $k$-alternating length, while for the other half the two lengths are equal. Therefore $E_{n}(zs_k)=E_{n}(as_k)+1/2$. Hence  we have:

\begin{lemma}\label{L:alternatingSimilar2zigzagging}
The formula \eqref{F:averagekalternating} is equivalent to the formula
 \begin{equation}\label{F:averagekzigzagging}
 E_{n}(zs_k)=\frac{2(n-k)+4}{3}.
 \end{equation}
\end{lemma}

Let us take a look at the $k=1$ case of the proof to get some ideas about our proof.
In this case, the zigzagging length of $w$ is equal to the number of its peaks and valleys, where $w_i$ is a peak (respectively a valley) if it is greater (respectively less) than its one or two neighbors. We see that $w_1$ and $w_n$ each is a peak or a valley. With a little thought, one sees that the probability that $w_i$ is a peak or a valley is $2/3$ when $1<i<n$. Now we see that $E_n(zs_1)=1+(n-2)\times \frac23+1=\frac{2n+2}3$, in agreement with \eqref{F:averagekzigzagging}. (The author learned this proof from Richard Stanley, who learned it from Miklos B\'ona. See Section 4 of \cite{S1}.)

Our proof is similar to this argument. We first define the $k$-peaks and $k$-valleys of a permutation, which are the original peaks and valleys when $k=1$. We prove that the $k$-zigzagging length of a permutation is equal to the number of its $k$-peaks and $k$-valleys. Then we count the probability that a number $j$ is a $k$-peak in a permutation. Finally, we prove formula \eqref{F:averagekzigzagging} which is equivalent to \eqref{F:averagekalternating}.

\section{$k$-peaks and $k$-valleys}
\begin{definition}
Let $w=w_1w_2\dotsm w_n\in\mathfrak{S}_n$ and $n>k\geq1$. We call a section $w_sw_{s+1}\dotsm w_t$ in $w$ a \emph{$k$-up} (respectively a \emph{$k$-down}) if $s<t$ and $w_t-w_s\geq k$ (respectively $w_s-w_t\geq k$).
We say a section $w_iw_{i+1}\dotsm w_j$ ($i<j$) of $w$ is \emph{$k$-ascending} if it satisfies the following:
\begin{enumerate}
\renewcommand*\labelenumi{[\theenumi]}
\item $w_i=\text{min}\{w_i,w_{i+1},\dots,w_j\}$, $w_j=\text{max}\{w_i,w_{i+1},\dotsc,w_j\}$;
\item $w_j-w_i\geq k$; i.e., $w_i\dotsm w_j$ is a $k$-up;
\item if $i\leq s<t\leq j$ then $w_s-w_t<k$; i.e., there is no $k$-down in $w_i\dotsm w_j$.
\end{enumerate}
 If moreover $w_i\dotsm w_j$ is not contained in another $k$-ascending section, we call it a \emph{maximal $k$-ascending} section. In this case, we call $w_i$ a \emph{$k$-valley} of $w$ and $w_j$ a \emph{$k$-peak} of $w$.
\end{definition}
Similarly, we define $k$-down, $k$-descending, and maximal $k$-descending. For a maximal $k$-descending section $w_i\dotsm w_j$ of $w$ we also call $w_i$ a $k$-peak of $w$ and $w_j$ a $k$-valley of $w$.

\begin{example}\label{E:1&n}
Let $w=w_1w_2\dotsm w_n\in\mathfrak{S}_n$. 
We see that if $1\leq j\leq k$, then the number $j$ is not a $k$-peak in $w$.
\end{example}

\begin{example}
Consider the permutation $w=214386759\in \mathfrak{S}_9$. We see that the number $2$ is not in a maximal $3$-ascending section or a maximal $3$-descending section. 
The sections $1438$ and $59$ are maximal $3$-ascending sections, while $8675$ is a maximal $3$-descending section. Finally, $1859$ is a longest $3$-zigzagging subsequence of $w$.
\end{example}

 This example suggests that a permutation can be viewed as a chain of alternating maximal $k$-ascending sections and maximal $k$-descending sections. The link points are those $k$-valleys and $k$-peaks. It is possible, however, that a beginning section or a ending section is not covered by this chain. Most importantly, we also see that the subsequence formed by the $k$-peaks and $k$-valleys is a longest $k$-zigzagging subsequence of $w$ (see Proposition \ref{P:zigzagging&peaks}). We will only need to count the total number of the $k$-peaks, because the total number of $k$-peaks of all permutations is equal to that of the $k$-valleys, which can be seen applying the swapping map $I$.

We have the following properties to prolong a $k$-ascending section. Using the swapping map $I$, one finds similar properties for a $k$-descending section.
\begin{lemma}\label{L:prolonging}
Let a section $w_i\dotsm w_j$ in $w=w_1\dotsm w_n$ be $k$-ascending.
\begin{enumerate}
\renewcommand*\labelenumi{(\theenumi)}
\item If there is a $t>j$ with $w_j<w_t$ and no $k$-down in $w_j\dotsm w_t$ then the $k$-ascending section $w_i\dotsm w_j$ can be prolonged from the right, i.e., there is a $j<t'\leq t$ such that $w_i\dotsm w_j\dotsm w_{t'}$ is $k$-ascending;
\item If there is a $s<i$ with $w_s<w_i$ and no $k$-down in $w_s\dotsm w_i$ then the $k$-ascending section $w_i\dotsm w_j$ can be prolonged from the left, i.e.,there is an $s\leq s'<i$ such that $w_{s'}\dotsm w_i\dotsm w_j$ is $k$-ascending.
\end{enumerate}
\end{lemma}
\begin{proof}
For the first statement, take $w_{t'}=\text{max}\{w_j, w_{j+1},\dotsc,w_t\}$. It is easy to verify that $w_i\dotsm w_j\dotsm w_{t'}$ is a desired $k$-ascending section. The second statement is completely analogous.
\end{proof}

The following property says that a $k$-up contains a $k$-ascending section. There is a similar fact for a $k$-down.
\begin{lemma}\label{L:up2ascending}
Let $(w_i,w_j)$ be a $k$-up. Let $i\leq i'<j'\leq j$ such that $w_{i'}\dotsm w_{j'}$ is a shortest (i.e., $|i'-j'|$ is minimal) $k$-up.
Then $w_{i'}\dotsm w_{j'}$ is a $k$-ascending section.
\end{lemma}
\begin{proof}
This can easily be verified by definition.
\end{proof}

\begin{lemma}\label{L:intersection}
The intersection of a maximal $k$-ascending section and a maximal $k$-descending section is empty or a one-element set. Two distinct maximal $k$-ascending sections do not intersect.
\end{lemma}
\begin{proof}
The first statement is easy by considering the maximum and minimum of the two sections.

The second statement follows from Lemma \ref{L:prolonging}.
\end{proof}

The following result together with Lemma \ref{L:up2ascending} tells us that every permutation $w$ is covered by its maximal $k$-ascending sections and maximal $k$-descending sections, except possibly a beginning section and/or an ending section of $w$.
\begin{lemma}
 Let $\mathbf{\gamma}=w_iw_{i+1}\dotsm w_j$ and $\mathbf{\delta}=w_{i'}w_{i'+1}\dotsm w_{j'}$ each be a maximal $k$-ascending section or a maximal $k$-descending section.
  If $j<i'$ then there is a $k$-up or a $k$-down in $w_jw_{j+1}\dotsm w_{i'}$.
\end{lemma}
\begin{proof}
If there is no $k$-up or $k$-down in $w_j\dotsm w_{i'}$, Lemma \ref{L:prolonging} will always allow us to prolong one of the two sections $\mathbf{\gamma}$ and $\mathbf{\delta}$, a contradiction to the maximality of $\mathbf{\gamma}$ and $\mathbf{\delta}$.

For example, let us consider the case that both $\mathbf{\gamma}$ and $\mathbf{\delta}$ are maximal $k$-ascending (and there is no $k$-down or $k$-up in $w_j\dotsm w_{i'}$). Then $w_i<w_{i'}$. (Otherwise, $w_j\dotsm w_{i'}$ is already a $k$-down as $w_j-w_{i'}> w_j-w_{i}\geq k$.) Moreover, there is no $k$-down in $w_i\dotsm w_{i'}$. Thus $w_{i'}\dotsm w_{j'}$ can be prolonged from the left by Lemma \ref{L:prolonging}.
\end{proof}
\begin{proposition}\label{P:zigzagging&peaks}
The subsequence of a permutation formed by the $k$-peaks and $k$-valleys is a longest $k$-zigzagging subsequence. Thus the average $k$-zigzagging length of permutations is two times the average number of $k$-peaks of permutations.
\end{proposition}
\begin{proof}
Let $w_{i_1}w_{i_2}\dotsm w_{i_s}$ be the subsequence formed by the $k$-peaks and $k$-valleys of $w$. Let $\gamma_r=w_{i_{r}}\dotsm w_{i_{r+1}}$ ($r=1,2,\dotsc,s-1$). We see that $w$ is a union of these $s+1$ sections $\gamma_0,\gamma_1, \dotsc, \gamma_{s-1},\gamma_s$, where $\gamma_1,\dotsm, \gamma_{s-1}$ is an alternating sequence of maximal $k$-ascending sections and maximal $k$-descending sections. (The (beginning) section of $w$, $\gamma_0=w_1\dotsm w_{i_1}$, is a single element if $i_1=1$. The (ending) section of $w$, $\gamma_s=w_{i_s}\dotsm w_n$, is a single element if $i_s=n$.) To form a $k$-zigzagging subsequence of $w$, one can take at most one element from each of $\gamma_0$ and $\gamma_s$. One can take at most two elements from each of $\gamma_1,\dotsm, \gamma_{s-1}$; but to take two elements from each of $\gamma_t,\gamma_{t+1}$, one has to take the link point $w_{i_{t+1}}$. Thus we see that taking the $k$-peaks and $k$-valleys is one way to have the maximum length of $k$-zigzagging subsequence.

The second statement now follows because the total number of $k$-peaks of all permutations is equal to that of $k$-valleys.
\end{proof}

\section{A characterization of $k$-peaks and the proof of the theorem}
We will need the following characterization of $k$-peaks.
\begin{proposition}\label{P:kpeakcharacterization}
Let $w=w_1\dotsm w_n\in \mathfrak{S}_n$, $i\in[1,n]$ and $1\leq k\leq n-1$. Then $w_i$ is a $k$-peak if and only if it satisfies the following two properties.
\begin{enumerate}
\renewcommand*\labelenumi{(\theenumi)}
\item If there is an $s>i$ with $w_s>w_i$, then there is a $k$-down $w_i\dotsm w_j$ in $w_i\dotsm w_s$.
\item If there is an $s<i$ with $w_s>w_i$, then there is a $k$-up $w_j\dotsm w_i$ in $w_s\dotsm w_i$.
\end{enumerate}
\end{proposition}

\begin{remark}
 (1) Note that if $w_i=n$ than it satisfies these two properties for all positive integers $k$. Therefore the number $n$ appears as a $k$-peak for all $1\leq k\leq n-1$. (2) By this proposition, a $k$-peak is also a $k'$-peak if $1\leq k'\leq k\leq n-1$.
\end{remark}

\begin{proof}[Proof of Proposition \ref{P:kpeakcharacterization}]
Proof of ``only if": Let $w_i$ be a $k$-peak. Then it is the ending of a maximal $k$-ascending section and/or the beginning of a $k$-descending section.
Let us consider the case that it is the ending of a maximal $k$-ascending section $w_{i'}\dotsm w_i$; the other case can be done similarly.

First $w_i$ satisfies the second property. Now assume that it does not satisfy the first property. Then we can take the minimum $s$ such that $s>i$, $w_s>w_i$ and there is no $k$-down $w_i\dotsm w_j$ in $w_i\dotsm w_s$. Then $w_i>w_{s'}$ for $i<s'<s$ by the minimality of $s$. Therefore there is no $k$-down in $w_i\dotsm w_s$. (Because if $w_{j'}\dotsm w_j$ is a $k$-down in $w_i\dotsm w_s$, then so is $w_i\dotsm w_j$ as $w_i>w_{j'}$). By Lemma \ref{L:prolonging} we can prolong the maximal $k$-ascending section $w_{i'}\dotsm w_i$ from the right, a contradiction.

Proof of ``if": First there is at least one $k$-down $w_i\dotsm w_j$ or one $k$-up $w_j\dotsm w_i$ (no matter whether $w_i$ equals $n$ or not). Let us prove the case when there is a $k$-up $w_j\dotsm w_i$; the other case is proved similarly. Let $w_t$ be the closest element to $w_i$ (so $|i-t|$ is minimum) such that $w_t\dotsm w_i$ is a $k$-up. We show in the following that $w_t\dotsm w_i$ is $k$-ascending.

 First, $w_t$ is the minimum in $\{w_t,\dotsc, w_i\}$ by the choice of it. Also $w_i$ is the maximum in $\{w_t,\dotsc, w_i\}$. Otherwise, let $w_s$ in $w_t\dotsm w_i$ be greater than $w_i$; thus there is a $k$-up $w_{s'}\dotsm w_i$ in $w_s\dotsm w_i$. This $w_{s'}$ is closer to $w_i$ than $w_t$ is, contradicting to the choice of $w_t$. Second, $w_t\dotsm w_i$ is known to be a $k$-up. Third, there is no $k$-down in $w_t\dotsm w_i$. Otherwise, let $w_r\dotsm w_s$ be a $k$-down in $w_t\dotsm w_i$. Then $w_i-w_s>w_r-w_s\geq k$ and thus $w_s\dotsm w_i$ is a $k$-up and $w_s$ is closer to $w_i$ than $w_t$ is, a contradiction.

Now as $w_t\dotsm w_i$ is a $k$-ascending section; it is thus contained in a maximal $k$-ascending section $w_{t'}\dotsm w_{i'}$.  If $i'>i$, then $w_{i'}>w_i$,  and thus there is a $k$-down $w_i\dotsm w_r$ in $w_i\dotsm w_{i'}$ (by the first property), which contradicts the fact that $w_{t'}\dotsm w_{i'}$ is a (maximal) $k$-ascending section. Therefore $i'=i$ and hence $w_i$ is a $k$-peak, as desired.
\end{proof}

Now we apply Proposition \ref{P:kpeakcharacterization} to find the probability that a number $j$ appears as a $k$-peak in a permutation in $\mathfrak{S}_n$.
 For instance, by this proposition, we know that the probability of $n$ being a $k$-peak is $1$.
\begin{proposition}\label{P:pnki}
Let $1\leq j\leq n$ and $1\leq k\leq n-1$.
 Let $p_{n,k}(j)$ be the probability that $j$ is a $k$-peak of a randomly selected permutation in $\mathfrak{S}_n$. We have
\begin{displaymath}
   p_{n,k}(j) = \left\{
     \begin{array}{lr}
       0 & \text{ if } j\leq k\\
       \frac{(j-k)(j-k+1)}{(n-k)(n-k-1)} & \text{ if } j>k.
     \end{array}
   \right.
\end{displaymath}
\end{proposition}
\begin{proof}
The case $j\leq k$ is known by Example \ref{E:1&n} or by Proposition \ref{P:kpeakcharacterization}.

Let us consider the case $j>k$. We partition the set $[1,n]-\{j\}$ into three subsets:
\begin{align*}
&A=\{l:1\leq l\leq j-k\}\\
&B=\{l:j-k+1\leq l\leq j-1\}\\
&C=\{l:j+1\leq l\leq n\}.
\end{align*}
To form a permutation, let us first arrange $A\cup\{j\}$ on a row $a_1a_2\dotsm a_{j-k+1}$, then we insert the elements from the set $B\cup C$ one by one into this row. We first insert the number $j+1$ into $a_1a_2\dotsm a_{j-k+1}$. There are $j-k+2$ positions to put: put it to the left of $a_1$, put it between $a_1$ and $a_2$, put it between $a_2$ and $a_3$, on and on, and put it to the right of $a_{j-k+1}$. We form a new row with $j+k+2$ elements. Then we put the number $j+2$ into this new row, and there are $j-k+3$ positions to do this. Keep doing this until we exhaust all elements in $C$; then do elements from $B$.

We see that all permutations can be obtained this way. But to make $j$ a $k$-peak, it is sufficient and necessary that we do not put any element from $C$ next to $j$. This is because Proposition \ref{P:kpeakcharacterization} tells us that between $j$ and an element from $C$ there should be at least an element from $A$. The insertion of elements from $B$ will not change the property that $j$ is a $k$-peak or not.

 Therefore when first adding $j+1$, there are $j-k$ \textit{right} positions out of the $j-k+2$ positions to put it.  When adding $j+2$, there are $j-k+1$ \textit{right} ways out of the $j-k+3$ ways to do so. So on and so forth, until when adding $n$, there are $n-k-1$ \textit{right} ways out of the $n-k+1$ ways to do so. Therefore the probability of $j$ being a $k$-peak is as follows:
\begin{align*}
p_{n,k}(j)&=\frac{j-k}{j-k+2}\times\frac{j-k+1}{j-k+3}\times\dotsm \times\frac{n-k-1}{n-k+1}\\
          &=\frac{(j-k)(j-k+1)}{(n-k)(n-k+1)}.
\end{align*}
\end{proof}

\begin{proof}[Proof of Theorem \ref{T:main}]
As the probability of $j$ being a $k$-peak in a permutation $w\in \mathfrak{S}_n$ is $p_{n,k}(j)$,
the average number of $k$-peaks of a permutations in $\mathfrak{S}_n$ is $\sum_{j=1}^n p_{n,k}(j)$. 
By Propositions \ref{P:zigzagging&peaks} and \ref{P:pnki}, we have
\begin{align*}
E_n(zs_k)&=2\sum_{j=1}^n p_{n,k}(j)\\
         &=2\sum_{j=k+1}^n\frac{(j-k)(j-k+1)}{(n-k)(n-k+1)}\\
         &=\frac{2(n-k)+4}{3}.
\end{align*}
This is formula \eqref{F:averagekzigzagging}, which is equivalent to \eqref{F:averagekalternating} by Lemma \ref{L:alternatingSimilar2zigzagging}.
\end{proof}

\centerline{\bf Acknowledgments}
The author gratefully acknowledges Professor Richard Stanley for his comprehensive help on this work.
 He also thanks M.I.T. for hospitality and the China Scholarship Council for the support during the work.
This work is partially supported by NSFC grant  \#11271138.
\bibliographystyle{amsalpha}

\begin{thebibliography}{99}
\bibitem{Arm} D. Armstrong, Enumerative Combinatorics Problem Session, in Oberwolfach Report No. 12/2014,
(March 2-8, 2014).

\bibitem{PP} I. Pak, R. Pemantle, On the longest k-alternating subsequence, arXiv:1406.5207 [math.CO].

\bibitem{S1} R. Stanley, Longest alternating subsequences of permutations, \emph{Michigan Math. J.} \textbf{57} (2008), 675--687.

\bibitem{S2} R. Stanley, Increasing and decreasing subsequences and their variants, in \emph{Proc. ICM Madrid,}
Vol. I, EMS, Z\"urich, 2007, 545--579.
\end{thebibliography}

\end{document}